\documentclass[11pt]{amsart}
\usepackage{url, amssymb, enumerate, upref, graphicx}
\newtheorem{theorem}{Theorem}[section]
\newtheorem{lemma}[theorem]{Lemma}
\newtheorem{proposition}[theorem]{Proposition}
\newtheorem{corollary}[theorem]{Corollary}
\theoremstyle{definition}
\newtheorem{definition}[theorem]{Definition}
\newtheorem{example}[theorem]{Example}
\newtheorem{question}[theorem]{Question}
\theoremstyle{remark}
\newtheorem{remark}[theorem]{Remark}
\numberwithin{equation}{section}
\newcommand{\notimplies}{%
  \mathrel{{\ooalign{\hidewidth$\not\phantom{=}$\hidewidth\cr$\implies$}}}}
\newcommand{\onto}{\xrightarrow{\textrm{onto}}}
\DeclareMathOperator{\re}{Re}
\DeclareMathOperator{\im}{Im}
\DeclareMathOperator{\diam}{diam}
\DeclareMathOperator{\dist}{dist}
\DeclareMathOperator{\disp}{disp}
\DeclareMathOperator{\dom}{dom}
\DeclareMathOperator{\sign}{sign}
\DeclareMathOperator{\med}{med}
\DeclareMathOperator{\Lip}{Lip}
\DeclareMathOperator{\dH}{d_H}

\title{Lipschitz means and mixers on metric spaces}

\author{Leonid V. Kovalev}
\address{215 Carnegie, Department of Mathematics, Syracuse University, Syracuse, NY 13244, USA}
\email{lvkovale@syr.edu}

\subjclass[2020]{Primary 51F30; Secondary 30L10, 54B20, 54C15, 54E40} 

\keywords{Metric space, Lipschitz mean, Lipschitz mixer,  bounded turning, quasisymmetric map, quasicircle}

\begin{document}
\baselineskip6mm
\maketitle

\begin{abstract} The standard arithmetic measures of center, the mean and median, have natural topological counterparts which have been widely used in continuum theory. In the context of metric spaces it is natural to consider the Lipschitz continuous versions of the mean and median. We show that they are related to familiar concepts of the geometry of metric spaces: the bounded turning property, the existence of quasisymmetric parameterization, and others. 
\end{abstract}

\section{Introduction} 

A \emph{mean} on a metric space $X$ is a continuous map $\mu\colon X^2\to X$ that is symmetric and idempotent: $\mu(a, b)=\mu(b, a)$ and $\mu(a, a)=a$ for all $a,b\in X$. The existence of means is a classical subject in continuum theory, initiated by Aumann~\cite{Aumann}; see~\cite[\S76]{IllanesNadler} for a survey. 
A \emph{Lipschitz mean} is also required to be Lipschitz continuous, thus bringing the metric on $X$ into play. While it is easy to see that every topological arc admits a mean, the corresponding question for Lipschitz means is not trivial even for rectifiable planar arcs. Indeed, we do not have a complete description of the metric arcs which admit a Lipschitz mean. The picture is more complete for a related notion of Lipschitz mixers. 
A \emph{mixer}~\cite{vanMill1} is a continuous map $\sigma\colon X^3\to X$ with the following \emph{absorption property}:   
\begin{equation}\label{eq-absorption}
\sigma(a, a, b)=\sigma(a, b, a)=\sigma(b, a, a)=a \quad \forall a, b\in X.    
\end{equation}
In the context of metric spaces, Lipschitz mixers are of interest because their existence nearly characterizes the Lipschitz retracts of Euclidean spaces (\cite{Hohti} and ~\cite[Theorem 2.12]{HeinonenLip}). We will prove the following result (see~\S\ref{sec-definitions} for definitions): 

\begin{theorem}\label{thm-BT-char-mix}
A metric arc admits a Lipschitz mixer if and only if it has bounded turning.
\end{theorem}

Theorem~\ref{thm-BT-char-mix} is somewhat unexpected because the bounded turning property is invariant under quasisymmetric maps (Definition~\ref{def:qs-map}) which have no obvious reason to preserve Lipschitz mixers. A similar statement is true for metric circles, provided that the mixer is understood in a local sense (Definition~\ref{def-local-mixer}).  

\begin{theorem}\label{intro-quasicircles-mixers}
A metric circle admits a local Lipschitz mixer if and only if it has bounded turning.
\end{theorem}

If a metric circle (resp. arc) is embedded in a Euclidean space, its bounded turning property is equivalent to being a quasisymmetric image of $S^1$ (resp. $[0, 1]$)~\cite{TukiaVaisala}. 

In general, a path-connected metric space with a Lipschitz mixer has bounded turning but the converse is false (Section~\ref{sec-metric-intervals}). 
In contrast, the existence of a Lipschitz mean does not imply the bounded turning property, even for spaces homeomorphic to $\mathbb R$. A class of counterexamples is provided by the following theorem.  

\begin{theorem}\label{two-quasiarcs}
Suppose that $f\colon \mathbb R\to\mathbb R$ is a continuous even function which is increasing for $x>0$. Then its graph $\Gamma=\{(x,f(x))\colon x\in\mathbb R\}\subset \mathbb R^2$ has a Lipschitz mean. 
\end{theorem}

In Section~\ref{sec-existence-means} we will find  obstructions to the existence of Lipschitz means, which imply that even some rectifiable planar arcs do not admit such a mean. Section~\ref{sec-mixers-circles} concerns the local forms of mixers and means, which are natural in the context of spaces with nontrivial topology. The paper concludes with open problems in Section~\ref{sec-remarks}. 

\section{Definitions and preliminary results}\label{sec-definitions}

When $(X, d)$ is a metric space, its Cartesian powers $X^n$ are equipped with the metric $d(a, b)=\sum_{k=1}^n d(a_k, b_k)$. A map $f\colon X\to Y$ is \emph{Lipschitz} if there exists a constant $L$ ($L=\Lip f$) such that $d_Y(f(a), f(b))\le Ld_X(a,b)$ for all $a,b\in X$. We will often use the fact that a map $f\colon X^n\to Y $ satisfies the Lipschitz condition with a constant $L$ if and only if it does so with respect to each variable separately. In the following, $B^n$ is the closed unit ball in $\mathbb R^n$ and $S^{n-1}$ is its boundary. 

\begin{theorem}\cite[Theorem 2.1]{Hohti}\label{thm-LC-Hohti} 
Suppose $X$ is a metric space with a Lipschitz mixer $\mu$, and $n\ge 2$. Then every Lipschitz map $f\colon S^{n-1}\to X$ extends to a Lipschitz map $F\colon B^n\to X$. Moreover, $\Lip F\le K \Lip f$ where $K$ depends only on $\Lip \mu$ and $n$. 
\end{theorem}

In the language of Lipschitz homotopy groups $\pi_m^{\Lip}$~\cite{HST, WengerYoung}, Theorem~\ref{thm-LC-Hohti} says that every metric space $X$ with a Lipschitz mixer has trivial groups $\pi_m^{\Lip}(X)$ for $m\ge 1$. The statement of~\cite[Theorem 2.1]{Hohti} includes the assumption that $X$ is compact, but it was only relevant because $\mu$ was assumed to be a local mixer (Definition~\ref{def-local-mixer}). A proof without the compactness assumption is given in~\cite[\S2.6]{HeinonenLip}. The statement of Theorem~\ref{thm-LC-Hohti} is false for $n=1$~\cite[\S4]{Hohti}. 

While a mean is required to be symmetric (otherwise $\mu(a, b)=a$ would qualify as a mean), this requirement is not imposed on mixers. However, explicit examples of mixers, such as the coordinate-wise median map
\begin{equation}\label{eq-coordinate-median}
\sigma(x, y, z) = (\med(x_i, y_i, z_i))_{i=1}^n,\quad x, y, z\in \mathbb R^n,    
\end{equation}
are often symmetric with respect to permutations of all three variables. Such mixers are called \emph{symmetric}. Symmetric mixers with additional algebraic properties are studied in the theory of median spaces  (e.g.,~\cite{vandeVel}). 

While the existence of a mean places strong constraints on the topology of the space~\cite{Eckmann, Sobolewski}, it does not force the space to be contractible. The following example of a non-simply-connected space with a Lipschitz mean is a special case of~\cite[Example 1.7]{Keesling1972}. 

\begin{example}\label{Keesling-example}  Let $H=\mathbb Z/(3\mathbb Z)$ and let $H^*$ be the space of all measurable functions $f\colon [0,1]\to H$, identifying the functions that agree a.e. This  space has a metric $d(f, g)=|\{x\in [0, 1]\colon f(x)\ne g(x)\}| $ where $|\cdot |$ is the Lebesgue measure. It admits a mean $\mu(f, g)=2f+2g$ which is $1$-Lipschitz in each argument. 

Let $X=H^*/H$ be the quotient of $H^*$ by the subspace of constant functions, i.e., the quotient by the isometric action $(h, f)\mapsto f+h$ where $h\in H$ and $f\in H^*$. Then $X$ is a metric space whose fundamental group is isomorphic to $H$~\cite[Theorem 2.3]{Keesling1973}. The mean $\mu$ induces a mean $\widetilde{\mu}$ on $X$, which is also Lipschitz because the action of $H$ on $H^*$ is isometric.  
\end{example}

Although the circle $S^1$ does not admit a  mean~\cite{Aumann}, the local forms of means and mixers are available on it. Mixers can be localized in two ways, depending on whether their domain is a set of the form 
\begin{equation}\label{eq-diagonal-nbhd}
\Delta_r(X^3) := \{(x_1, x_2, x_3)\in X^3\colon \max_{i<j} d(x_i, x_j)\le r\}    
\end{equation}
or of the form 
\begin{equation}\label{eq-diagonal-aug-nbhd}
\widetilde{\Delta}_r(X^3) := \{(x_1, x_2, x_3)\in X^3\colon \min_{i<j} d(x_i, x_j)\le r\}.    
\end{equation}
The former approach was taken in~\cite{Hohti}, the latter in~\cite{vanMill2} (see Lemma 2.3 in~\cite{vanMill2}). To clarify this distinction, we define \emph{local} and \emph{semilocal} mixers. 

\begin{definition}\label{def-local-mixer} 
A metric space $X$ admits a \emph{local Lipschitz mixer} if there exists $r> 0$ and a Lipschitz map $\sigma\colon \Delta_r(X^3)\to X$ with the absorption property~\eqref{eq-absorption}. 

If $\Delta_r(X^3)$ is replaced by $\widetilde{\Delta}_r(X^3)$ in the previous paragraph, we say that $\sigma$ is a \emph{semilocal Lipschitz mixer}. 
\end{definition}

Thus, a local mixer acts on three-point subsets of small diameter, while a semilocal mixer acts on three-point subsets of small minimal separation between points. 

A (Lipschitz) retract of a metric space $X$ is a subset $E\subset X$ for which there exists a (Lipschitz) continuous idempotent map $f\colon X\to E$. It is easy to see that Lipschitz mixers are inherited by Lipschitz retracts. The (semi)local mixers are also inherited by Lipschitz \emph{neighborhood} retracts. The proof of this fact is based on a simple property of Lipschitz mixers: 
\begin{equation}\label{eq-mixer-displacement}
d(\sigma(a, b, c), a)\le \Lip(\sigma)\,d(a, b)
\end{equation}
which follows from $\sigma(a, a, c)=a$. The same holds after any permutation of the arguments of $\sigma$, e.g., $d(\sigma(a, b, c), b)\le \Lip(\sigma)\,d(b, c)$.

\begin{lemma}\label{LNR-implies-mixer} 
Suppose that $X$ is a metric space with a semilocal (resp. local) Lipschitz mixer. If a closed set $E\subset X$ is a Lipschitz retract of an open set $U\subset X$ such that $\dist(E, X\setminus U)>0$, then $E$ admits a semilocal (resp. local) Lipschitz mixer. 
\end{lemma}

\begin{proof} Suppose $\sigma\colon \widetilde{\Delta}_r(X^3) \to X$ is a semilocal mixer. For sufficiently small $\rho\in (0, r)$, the mixer $\sigma$ maps $\widetilde{\Delta}_{\rho}(E^3)$ into $U$ by virtue of~\eqref{eq-mixer-displacement}.
Composing $\sigma$ with a Lipschitz retraction of $U$ onto $E$ yields a map from $\widetilde{\Delta}_{\rho}(E^3)$ to $E$ which is easily seen to be a semilocal Lipschitz mixer. 
The same proof works when $\sigma\colon {\Delta}_r(X^3) \to X$ is a local mixer. 
\end{proof}

The local form of Lipschitz means involves neighborhoods of the diagonal, $\Delta_r(X^2)=\{(a,b)\in X^2\colon d(a, b)\le r\}$. 

\begin{definition}\label{def-local-mean} 
A metric space $X$ admits a \emph{local Lipschitz mean} if there exists $r> 0$ and a Lipschitz map $\mu\colon \Delta_r(X^2)\to X$ such that $\mu(a,a)=a$ and $\mu(a, b)=\mu(b, a)$ for all $(a, b)\in \Delta_r(X^2)$.  
\end{definition}

Repeating the argument from the proof of Lemma~\ref{LNR-implies-mixer}, we obtain the following lemma.

\begin{lemma}\label{LNR-implies-mean} 
Suppose that $X$ is a metric space with a local Lipschitz mean. If a closed set $E\subset X$ is a Lipschitz retract of an open set $U\subset X$ such that $\dist(E, X\setminus U)>0$, then $E$ admits a local Lipschitz mean. 
\end{lemma}

\begin{definition}\label{def-metric-BT} A metric space $X$ has \emph{bounded turning}  if there exists a constant $C$ such that any two points $a, b\in X$ are contained in some compact connected set $E\subset X$ with $\diam E \le Cd(a, b)$. 
\end{definition}

\begin{definition}\label{def-metric-doubling} A metric space $X$ is \emph{doubling}  if there exists a constant $N$ such that any ball in $X$ can be covered by at most $N$ balls of half the radius. 
\end{definition}

By a \emph{metric interval} we mean a metric space $(\Gamma, d)$ that is homeomorphic to a nontrivial interval of the real line $\mathbb R$. Such spaces form three homeomorphism classes: lines, rays, and arcs, based on being homeomorphic to $\mathbb R$, $[0, \infty)$, or $[0,1]$. A \emph{metric circle} is a metric space homeomorphic to the circle $S^1$. In the context of metric intervals and circles, the bounded turning property amounts to an upper bound on the diameter of a subarc between any two given points $a, b$. 

By a theorem of Tukia and V\"ais\"al\"a~\cite{TukiaVaisala}, a metric circle $\Gamma$ admits a quasisymmetric parametrization by $S^1$ if and only if it has both bounded turning property and the doubling property. The same holds for quasisymmetric parameterization of metric arcs by $[0, 1]$. 

\begin{definition}\label{def:qs-map}  A homeomorphism $f\colon X\onto Y$ is \emph{quasisymmetric} if there exists a homeomorphism $\eta \colon [0, \infty) \to [0, \infty)$ such that for any three distinct points $x, u, v$ in $X$ we have
\begin{equation}\label{eq-def-qs-map}
\frac{d_Y(f(x), f(u))}{d_Y(f(x), f(v))}
\le \eta\left( \frac{d_X(x, u)}{d_X(x, v)}\right). 
\end{equation}
\end{definition}

For example, if there exists $L$ such that $\Lip f\le L$ and $\Lip f^{-1} \le L$ (i.e., $f$ is \emph{bi-Lipschitz}), then~\eqref{eq-def-qs-map} holds with $\eta(t)=L^2t$. An example of a quasisymmetric map that is not bi-Lipschitz is~\cite[\S6.5]{Vaisala1990} 
\begin{equation}\label{eq-power-map}
f(z)=|z|^{\alpha-1}z,\quad z\in \mathbb C,\quad  (f(0)=0)        
\end{equation}
where $\alpha>0$ and $\alpha\ne 1$. 

Lipschitz means and mixers are obviously preserved by bi-Lipschitz maps. More generally, they are preserved by quasihomogeneous maps, which lie between the  bi-Lipschitz and quasisymmetric classes.

\begin{definition}\label{def:qh-map} \cite{GhamsariHerron,HerronMeyer}   A homeomorphism $f\colon X\onto Y$ is \emph{quasihomogeneous} if there exists a homeomorphism $\eta \colon [0, \infty) \to [0, \infty)$ such that for any four distinct points $x_1, \dots, x_4$ in $X$ we have
\begin{equation}\label{eq:qh-map}
\frac{d_Y(y_1, y_2)}{d_Y(y_3, y_4)}
\le \eta\left( \frac{d_X(x_1, x_2)}{d_X(x_3, x_4)}\right) 
\end{equation}
where $y_k=f(x_k)$, $k=1, \dots, 4$. 
\end{definition}

\begin{proposition}\label{prop:qhmaps}
If $f\colon X\onto Y$ is a quasihomogeneous map and $X$ admits a Lipschitz mixer, then $Y$ admits a Lipschitz mixer as well. When $X$ is compact, this also holds for local or semilocal mixers.
\end{proposition}

\begin{proof} Let $\sigma$ be a mixer on $X$ (possibly a local or semilocal one). For $(a, b, c)\in \dom \sigma$ we define $\tau(f(a), f(b), f(c)) = f(\sigma(a, b, c))$.  Clearly, $\tau$ is a global mixer on $Y$ when $\sigma$ is a global mixer on $X$. When $X$ is compact, both $f$ and $f^{-1}$ are uniformly continuous, which ensures that the image of $\Delta_r(X^3)$ under the map $(a, b, c)\mapsto (f(a), f(b), f(c))$ contains  $\Delta_\rho(Y^3)$ for some $\rho>0$. This ensures that when $\sigma$ is a local mixer, so is $\tau$. Same reasoning applies to the semilocal case.

To verify the Lipschitz property of $\tau$, consider two triples $(a, b, c)$ and $(a', b', c')$ in $\dom \sigma$.
Up to the reordering of variables, we have $d(a, a')\ge \max(d(b, b'), d(c, c'))$. Then the points  $w=\sigma(a,b,c)$ and $w'=\sigma(a',b',c')$ satisfy $d(w,w')\le 3Ld(a,a')$ where $L=\Lip\sigma$. The quasihomogeneity of $f$ implies $d(f(w),f(w'))\le \eta(3L) d(f(a),f(a'))$. Therefore, $\tau$ satisfies the Lipschitz condition with the constant $\eta(3L)$. 
\end{proof}

Proposition~\ref{prop:qhmaps} shows that every quasihomogeneous image of $\mathbb R^n$ admits a Lipschitz mixer, and every quasihomogeneous image of $S^n$ admits a semilocal Lipschitz mixer. Such metric spaces were studied in~\cite{Bishop01, Freeman11, Freeman13}.

\section{Existence of Lipschitz mixers}\label{sec-metric-intervals} 

Every metric interval $\Gamma$ has a total order induced by a homeomorphism with an interval in $\mathbb R$. Using such an order, we introduce three natural maps on $\Gamma^2$ or $\Gamma^3$:  
\begin{equation}\label{eq-min-med}
    (a, b)\mapsto \min(a, b), \quad (a, b)\mapsto \max(a, b), \quad 
    (a, b, c) \to \med(a, b, c) 
\end{equation}
where $\med$ is the median of three elements, that is, 
\begin{equation}\label{eq-med-from-max-min}
\med(a,b,c)=\max(\min(a,b), \min(b,c), \min(a, c)). 
\end{equation}

Each of the maps in~\eqref{eq-min-med} provides a simple characterization of the metric intervals of bounded turning. 

\begin{proposition}\label{BT-char}
The following are equivalent properties of a metric interval $(\Gamma, d)$: 
\begin{enumerate}[(i)]
    \item\label{BT1} $\Gamma$ has bounded turning;
    \item\label{BT2} the maps $\min$ and $\max$ are Lipschitz on $\Gamma^2$;
    \item\label{BT3} the median map $\med$ is Lipschitz on $\Gamma^3$.
        
\end{enumerate}
The implicit constants in statements~\eqref{BT1}, \eqref{BT2}, \eqref{BT3} depend only on each other. 
\end{proposition}

\begin{proof} Since the maps in~\eqref{eq-min-med} are symmetric, it suffices to consider their Lipschitz property with respect to the first argument.  

\eqref{BT1}$\implies$\eqref{BT2} Suppose $\Gamma$ has bounded turning with a constant $C$. Take three points $a, b, a' \in \Gamma$ and let 
$\rho = d(\min(a,b), \min(a', b))$. 
If $b$ does not separate $a$ from $a'$, then either $\rho=0$ or $\rho=d(a, a')$. If $b$ does separate $a$ from $a'$, then  
either $\rho=d(a,b)$ or $\rho=d(a', b)$, and both of those are bounded by $C d(a, a')$. Thus, $\rho \le C d(a, a')$ 
holds in every case. 

\eqref{BT2}$\implies$\eqref{BT3} 
This is immediate from~\eqref{eq-med-from-max-min}. 

\eqref{BT3}$\implies$\eqref{BT1} Let $L$ be the Lipschitz constant of the median map with respect to the first argument. Given two points $a, b$ on $\Gamma$, let $E$ be the subarc of $\Gamma$ between $a$ and $b$. For any point $c\in E$ we have $\med(a,b,c)=c$ and  $\med(b,b,c)=b$. Therefore 
$d(b,c)\le L d(a,b)$. Since $c\in E$ was arbitrary, it follows than $\diam E\le 2L d(a,b)$ and thus $\Gamma$ has  bounded turning. 
\end{proof}

\begin{proposition}\label{path-conn-BT} 
If a path-connected metric space $(X, d)$ admits a Lipschitz mixer, then it has bounded turning. 
\end{proposition}

\begin{proof} Given $a, b\in X$, let $\Gamma\colon [0, 1]\to X$ be a curve such that $\Gamma(0)=a$ and $\Gamma(0)=b$. Define $\gamma\colon [0, 1]\to X$ by $\gamma(t)=\sigma(a, b, \Gamma(t))$. Then $\gamma(0)=a$, $\gamma(1)=b$, and for every $t\in [0, 1]$ we have $d(\gamma(t), a)
\le\Lip(\sigma)\,d(a, b)$ by virtue of~\eqref{eq-mixer-displacement}. Thus $\diam \gamma\le 2 \Lip(\sigma)\, d(a, b)$.
\end{proof}

Theorem~\ref{thm-BT-char-mix} is a part of the following corollary.

\begin{corollary}\label{BT-char-mix}
The following are equivalent properties of a metric interval $(\Gamma, d)$: 
\begin{enumerate}[(i)]
    \item\label{BT1a} $\Gamma$ has bounded turning;
    \item\label{BT2a} $\Gamma$ admits a Lipschitz mixer;
    \item\label{BT3a} $\Gamma$ admits a symmetric Lipschitz mixer.
\end{enumerate}
\end{corollary}

\begin{proof} Proposition~\ref{path-conn-BT} shows that \eqref{BT2a}$\implies$\eqref{BT1a}. The implication \eqref{BT3a}$\implies$\eqref{BT2a} is trivial. Finally, \eqref{BT1a}$\implies$\eqref{BT3a} is a part of Proposition~\ref{BT-char} because the median map~\eqref{eq-min-med} is a symmetric mixer. 
\end{proof}

Since the bounded turning property is preserved by quasisymmetric maps, the existence of a Lipschitz mixer turns out to be a quasisymmetric invariant for metric intervals. However, Example~\ref{example-no-qs-mixer} will show that it is not quasisymmetrically invariant in general. Its construction will involve chain-connected spaces, which are defined next. 

\begin{definition}\label{def-chain-connected} A metric space $(X, d)$ is \emph{chain-connected} if for any two points $a, b\in X$ and any $\epsilon>0$ there exists a finite sequence $x_0, \dots, x_n$ with $x_0=a$, $x_n=b$, and $d(x_k, x_{k-1})<\epsilon$ for $k=1, \dots, n$. Such a sequence is called an \emph{$\epsilon$-chain}.
\end{definition}

A metric space is \emph{proper} if its closed bounded subsets are compact. 

\begin{proposition}\label{prop-chain-connected} Every  proper chain-connected metric space with a Lipschitz mixer is connected. 
\end{proposition}

\begin{proof} Let $X$ be such a space with a mixer $\sigma$, and let $L=\Lip \sigma$. Suppose to the contrary that there exists a subset $U\subset X$ which is both closed and open, with some points $a\in U$ and $b\in X\setminus U$. Let $B$ be the closed ball with center $a$ and radius  $L\,d(a, b)$. Given $\epsilon>0$, pick an $\epsilon$-chain $\{x_k\colon k=0, \dots, n\}$ from $a$ to $b$ and let $y_k=\sigma(a, b, x_k)$. For all $k$ we have $y_k\in B$ by virtue of~\eqref{eq-mixer-displacement}. Since $\{y_k\}$ is an $(L\epsilon)$-chain from $a$ to $b$ and is contained in $B$, it follows that $\dist(B\cap U, B\setminus U)\le L\epsilon$. But $\epsilon>0$ was arbitrary, contradicting the fact that $B\cap U$ and $B\setminus U$ are disjoint compact sets. 
\end{proof}

\begin{example}\label{example-no-qs-mixer} Let $E=\mathbb R\times \{0,1\}$ be the union of two parallel lines in $\mathbb R^2$. The set $E$ admits a Lipschitz mixer, namely the coordinate-wise median map~\eqref{eq-coordinate-median}. Fix $0<\alpha<1$ and let $f$ be the quasisymmetric map~\eqref{eq-power-map}. Then  $f(E)$ does not admit  a Lipschitz mixer. 
\end{example}

\begin{proof}  The set $f(E)$ is the union of the real axis $\Gamma_0$ with the curve $\Gamma_1$ that is asymptotic to $\Gamma_0$ because $\im f(x+i)\to 0$ as $|x|\to\infty$, $x\in \mathbb R$. Therefore, $f(E)$ is chain-connected. Since $f(E)$ is not connected, Proposition~\ref{prop-chain-connected} implies that $f(E)$ does not admit a Lipschitz mixer. 
\end{proof}

\begin{example}\label{example-grushin}
A quasisymmetric image of $\mathbb R^2$ (a \emph{quasiplane}) need not admit a Lipschitz mixer. A counterexample is the Grushin plane $\mathbb G$, a sub-Riemannian manifold which is a quasisymmetric image of $\mathbb R^2$~\cite{Meyerson}. By~\cite[Remark 8.4]{DejarnetteHLT} $\mathbb G$ has nontrivial Lipschitz fundamental group, and therefore admits no Lipschitz mixer. 
\end{example}

\section{Existence of Lipschitz means}\label{sec-existence-means}

In the literature on continuous means, the number of arguments is often taken to be $n\ge 2$~\cite{Aumann, Bacon1969, Eckmann}. By definition, an \emph{$n$-mean} is a continuous map $\mu\colon X^n\to X$ which is symmetric with respect to any permutation of its arguments and satisfies $\mu(x, \dots, x)=x$ for all $x\in X$. The existence of Lipschitz $n$-means is related to the Lipschitz clustering property, which is defined below. 

Given a metric space $X$ and a positive integer $n$, let $X(n)$ be the set of all nonempty subsets of $X$ with at most $n$ elements. The Hausdorff metric on $X(n)$ is introduced as
\begin{equation}\label{def-dh}
\dH(A, B) = \max\left(\sup_{a\in A} \dist(a, B), \sup_{b\in B} \dist(b, A)\right).
\end{equation}

The definition of the Hausdorff distance can be phrased in terms of the  \emph{displacement} of a map in a metric space, defined as 
$\disp(f)=\sup\{d(f(x), x)\colon x\in \dom f\}$. Indeed, $\dH(A, B)$ is the smallest number $\rho$ for which there exist maps $f\colon A\to B$ and $g\colon B\to A$ with $\disp(f)\le \rho$ and $\disp(g)\le \rho$.

\begin{definition} \cite{Kovalev2022} \label{def-lcp} A metric space 
$X$ has the \emph{Lipschitz clustering property} if there exist Lipschitz retractions $X(n)\to X(k)$ for all integers $n > k \ge 1$. 
\end{definition}

We usually write $X$ instead of $X(1)$ since these two spaces are naturally identified.  

\begin{lemma}\label{set-symm-means} A metric space $X$ admits a Lipschitz retraction $X(n)\to X$ if and only if it admits a Lipschitz $n$-mean $\mu$ with the additional property that 
\begin{equation}\label{eq-set-symm}
\mu(a_1, \dots, a_n) =  \mu(b_1, \dots, b_n)   \quad \text{whenever }\ \{a_1,\dots,a_n\}=\{b_1,\dots,b_n\}.
\end{equation}
\end{lemma}

\begin{proof} A Lipschitz retraction $r\colon X(n)\to X$ induces a map $X^n\to X$ by $\mu(a_1, \dots, a_n)=r(\{a_1, \dots, a_n\})$. It is clear than $\mu$ is a Lipschitz mean and ~\eqref{eq-set-symm} holds. 

Conversely, suppose that $\mu$ is an $L$-Lipschitz mean and~\eqref{eq-set-symm} holds. Any set $A\in X(n)$ can be written as $\{a_1,\dots, a_n\}$ by possibly repeating some of its elements. By virtue of~\eqref{eq-set-symm} the map $r(A):=\mu(a_1, \dots, a_n)$ is well defined. By the definition of a mean, $r(\{a\})=a$. Given any two sets $A, B\in X(n)$, let $\rho=\dH(A, B)$ and pick two maps $f\colon A\to B$ and $g\colon B\to A$ such that $\disp(f)\le \rho$ and $\disp(g)\le \rho$. Let $B'=f(A)$ and define $h\colon B\to B$ so that $h(b)=b$ when $b\in B'$ and $h(b)=f(g(b))$ otherwise. Note that $h(B)=B'$ and $\disp(h)\le 2\rho$. Writing $A=\{a_1, \dots, a_n\}$, we have 
\[d(r(A), r(B'))
=d(\mu(a_1, \dots, a_n), \mu(f(a_1), \dots, f(a_n)) \le nL\rho\]
by the Lipschitz property of $\mu$. Similarly, $d(r(B), r(B'))\le 2nL\rho$ using $B'=h(B)$. 
By the triangle inequality,
$d(r(A), r(B))\le  3nL\rho$ which shows that $r$ is Lipschitz in the Hausdorff metric. 
\end{proof}

In the special case $n=2$ the property~\eqref{eq-set-symm} reduces to $\mu$ being a symmetric function, which is a part of the definition of a mean. Thus, the existence of a Lipschitz retraction $X(2)\to X$ is equivalent to the existence of a Lipschitz mean on $X$. 
In view of this connection, the following result is a consequence of~\cite[Corollary 1]{Kovalev2022} which states that uniformly disconnected spaces have the Lipschitz clustering property. 

\begin{corollary}\label{unif-disconnect} Suppose that $(X, d)$ is a uniformly disconnected metric space; that is, there exists a constant $c>0$ such that any finite sequence $x_0, \dots, x_m$ in $X$ satisfies 
\[
\max_{1\le k\le m}d(x_k, x_{k-1}) \ge c\, d(x_0, x_m).
\]
Then $X$ admits a Lipschitz $n$-mean for every $n\ge 2$.
\end{corollary}

Proposition~\ref{BT-char} shows that every metric interval $\Gamma$ of bounded turning admits a Lipschitz mean, e.g., $(a, b)\mapsto \min(a, b)$.  However, the existence of a Lipschitz mean does not imply that $\Gamma$ has bounded turning. 
\begin{proposition}\label{two-quasiarcs-abstract}
Suppose that $d$ is a metric on $\mathbb R$ with the following properties: 
\begin{enumerate}[(a)]
\item\label{metric-ax1} the maps $x\to -x$ and $x\to |x|$ are Lipschitz from $(\mathbb R, d)$ to $(\mathbb R, d)$;
\item\label{metric-ax2} the half-line $[0, \infty)$ has bounded turning with respect to $d$;
\item\label{metric-ax3} there exists $C>0$ such that $d(-z,z)\le Cd(x,y)$ whenever $x\le -z \le z\le y$.
\end{enumerate} 
Then for every $n\ge 2$ the metric space $X=(\mathbb R, d)$ admits a Lipschitz $n$-mean, specifically
\begin{equation}\label{eq-n-lip-mean}
\mu(x_1, \dots, x_n) = \left(\min_{1\le k\le n} |x_k|\right)
\, \sign \max_{1\le k\le n} x_k 
\end{equation}
where $\sign x \in \{-1, 0, 1\}$ is the sign of real number $x$. 
\end{proposition}

\begin{proof} It suffices to prove that $\mu$ is Lipschitz with respect to $x_1$. Let $b=\min_{2\le k\le n} |x_k| $ and $c=\sign \max_{2\le k\le n} x_k$. The formula~\eqref{eq-n-lip-mean} reduces to the map $\nu(x)=\min(|x|, b)\max(\sign x, c)$ where we write $x$ for $x_1$. There are three cases based on the value of $c$.

If $c=1$, then $\nu(x)= \min(|x|, b)$ and this map is Lipschitz on $X$ by virtue of \eqref{metric-ax1}-\eqref{metric-ax2} and Proposition~\ref{BT-char}.

If $c=0$, then $b=0$ and therefore $\nu(x)=0$ for all $x\in X$.

Suppose $c=-1$. Then $\nu(x)=\min(|x|, b)\sign x$ which is a Lipschitz function on the intervals where $\sign x$ is constant. Also, $\nu(x)=x$ on the interval $[-b, b]$. Thus, it remains to estimate $d(\nu(x), \nu(x'))$ when $x$ and $x'$ have opposite signs and $\max(|x|,|x'|)>b$. If both $|x|$ and $|x'|$ are greater than $b$, then $\nu(\{x,x'\}) = \{-b,b\}$ and we have $d(\nu(x), \nu(x'))\le Cd(x, x')$ by virtue of~\eqref{metric-ax3}. If $|x|>b\ge |x'|$, then we obtain $d(x', -x')\le Cd(x, x')$ from~\eqref{metric-ax3}. Hence
\[
d(\nu(x), \nu(x')) = d(\nu(x), x')
\le Cd(x, x') + d(\nu(x), -x')
=Cd(x, x') + d(\nu(x), \nu(-x'))
\]
Since $x$ and $-x'$ lie in the same halfline where $\nu$ is known to be Lipschitz, the term $d(\nu(x), \nu(-x'))$ is bounded by a multiple of $d(x,-x')$. Finally, $d(x, -x')\le (C+1)d(x,x')$ by a combination of~\eqref{metric-ax3} and the triangle inequality. The proof is complete. 
\end{proof}

Theorem~\ref{two-quasiarcs} is a special case of the following result. 

\begin{corollary}\label{cor-two-quasiarcs}
Suppose that $f\colon \mathbb R\to\mathbb R$ is a continuous even function which is increasing for $x>0$. Then its graph $\Gamma=\{(x,f(x))\colon x\in\mathbb R\}\subset \mathbb R^2$ admits a Lipschitz retraction $\Gamma(n)\to \Gamma$ and therefore has a Lipschitz $n$-mean for every $n\ge 2$. 
\end{corollary}

\begin{proof} The curve $\Gamma$ is isometric to the real line equipped with the metric 
\[d(a, b) = 
\sqrt{(a-b)^2 + (f(a)-f(b))^2},\] which satisfies the assumptions of Proposition~\ref{two-quasiarcs-abstract}. The mean constructed in~\eqref{eq-n-lip-mean} satisfies~\eqref{eq-set-symm} and therefore defines a retraction $\Gamma(n)\to \Gamma$.
\end{proof}

\begin{remark}\label{rem-not-BT}
The curve $\Gamma$ in Corollary~\ref{cor-two-quasiarcs} need not have bounded turning, as is demonstrated by the cusp curve $y=\sqrt{|x|}$ and the parabola $y=x^2$.  
\end{remark}

Having obtained some sufficient conditions for the existence of a Lipschitz mean, we turn to necessary ones. 

\begin{lemma}\label{necessary-for-Lip-mean}
Suppose that $X$ is a path-connected space with an $L$-Lipschitz mean. Then for any two points $a,b\in X$ there is a curve $\gamma\colon [0, 2]\to X$ such that $\gamma(0)=a$, $\gamma(2)=b$, and 
\begin{equation}\label{eq-symm-curve}
d(\gamma(t), \gamma(t+1))\le L\, d(a, b),\quad 0\le t\le 1.     
\end{equation}
\end{lemma}

\begin{proof} Let $\mu$ be an $L$-Lipschitz mean on $X$. Given $a, b\in X$, choose a curve $\Gamma\colon [0, 1]\to X$ such that $\Gamma(0)=a$ and $\Gamma(1)=b$. Define another curve $\gamma$ as 
\[
\gamma(t)=\begin{cases}
\mu(\Gamma(t), a),\quad & 0\le t\le 1; \\ 
\mu(\Gamma(t-1), b),\quad & 1 <  t\le 2. 
\end{cases} 
\]
The map $\gamma$ is continuous at $t=1$ because $\mu$ is symmetric. The idempotence of $\mu$ implies $\gamma(0)=a$ and $\gamma(2)=b$. The Lipschitz property of $\mu$ implies~\eqref{eq-symm-curve}.
\end{proof}

The property~\eqref{eq-symm-curve} is weaker than bounded turning, and is not as geometrically intuitive. However, it leads to a usable necessary condition for the existence of a Lipschitz mean on some planar sets. Recall that for any point $z_0\in \mathbb C$ and any curve $\gamma\colon [\alpha, \beta] \to  \mathbb C\setminus \{z_0\}$ the complex argument $t\mapsto \arg (\gamma(t)-z_0) $ has a continuous branch and therefore the change of argument 
\[\Delta_{\arg}(\gamma-z_0):= \arg (\gamma(\beta)-z_0) - \arg (\gamma(\alpha)-z_0)
\]
is well defined. 

\begin{lemma}\label{argument-and-mean} 
Suppose $z_0\in \mathbb C$ and $\gamma\colon [0, 2]\to \mathbb C\setminus \{z_0\}$ is a curve. If  
\begin{equation}\label{eq-arg-min1}
\sup_{0\le t\le 1} |\gamma(t+1)-\gamma(t)| \le  \dist(z_0, \gamma)
\end{equation}
then $\Delta_{\arg}(\gamma-z_0)$ is within $2\pi/3$ of an integer multiple of $4\pi$.
\end{lemma}

\begin{proof} Let $h(t)=\arg (\gamma(t)-z_0)$. By virtue of~\eqref{eq-arg-min1}, for every $t\in [0, 1]$ the difference  $h(t+1)-h(t)$ is within $\pi/3$ of some multiple of $2\pi$. Since $h$ is continuous, there exists  $k\in \mathbb Z$ such that $|h(t+1)-h(t)-2\pi k|\le \pi/3$ for all $t\in [0, 1]$. Using this fact with $t=0$ and $t=1$, we obtain $|h(2)-h(0)-4\pi k|\le 2\pi/3$ as claimed. 
\end{proof}

Combining Lemma~\ref{necessary-for-Lip-mean} and Lemma~\ref{argument-and-mean} we arrive at the following corollary. 

\begin{corollary}\label{necessary-mean-planar}
Suppose that $E\subset \mathbb C$ is a path-connected set with an $L$-Lipschitz mean. If the points $z_0\in \mathbb C\setminus E$ and $a, b\in E$ are such that $L|a-b| \le \dist(z_0, E)$, then 
there is a curve $\gamma$ connecting $a$ to $b$ within $E$ for which $\Delta_{\arg}(\gamma-z_0)$ is within $2\pi/3$ of an integer multiple of $4\pi$. 
\end{corollary}

\begin{example}\label{long-arcs-no-mid} 
Fix $r>0$ and $T\in (\pi, 2\pi)$. If the set $E:=\{re^{it}\colon 0 \le t\le T\}$ has an $L$-Lipschitz mean, then $L > 1/(2\pi -T)$.

Indeed, if $L \le 1/(2\pi -T)$ then we can apply  Corollary~\ref{necessary-mean-planar} with $z_0=0$, $a=r$, $b=re^{iT}$, and obtain a contradiction because $T$ is not within $2\pi/3$ of an integer multiple of $4\pi$.  
\end{example}

By iterating Example~\ref{long-arcs-no-mid} we can construct a rectifiable planar arc without a Lipschitz mean. 

\begin{example}\label{ex-circles}
In the line segment $[0, 1]$, replace a small neighborhood of each point $2^{-n}$ with a circular arc $C_n$ in the upper halfplane with radius $2^{-n-2}$ and angular size $2\pi - n^{-1}$. Since these arcs are disjoint, the resulting curve $\Gamma$ is a Jordan arc. It is rectifiable because $\sum 2^{-n-2}<\infty$. However, $\Gamma$ admits no Lipschitz mean, as one can see by applying the reasoning in Example~\ref{long-arcs-no-mid} to each arc $C_n$. 
\end{example}

We saw in Corollary~\ref{BT-char-mix} that the existence of a Lipschitz mixer is a quasisymmetrically invariant property of  metric intervals. This invariance does not hold for Lipschitz means.

\begin{example}\label{ex-mean-not-qs-inv} Let $\gamma$ be the parabola $y=x^2$ in $\mathbb R^2$. Fix $0<\alpha<1/2$ and consider the quasisymmetric map $f$ from~\eqref{eq-power-map}. The curve $\gamma$ has a Lipschitz mean by Corollary~\ref{cor-two-quasiarcs} but $f(\gamma)$ does not. 
\end{example}

\begin{proof} The curve $\Gamma=f(\gamma)$ has a parametric equation 
\[\Gamma(x) = \frac{x+ix^2}{(x^2+x^4)^{(1-\alpha)/2}},\quad x\in \mathbb R,
\]
which shows that $|\re \Gamma(x)|\le |x|^{2\alpha-1}\to 0$ as $|x|\to \infty$. Consequently,  $|\Gamma(x)-\Gamma(-x)|\to 0$ as $|x|\to\infty$. Suppose that $\Gamma$ has an $L$-Lipschitz mean. Choose $x$ large enough so that $L |\Gamma(x)-\Gamma(-x)| \le \dist(i, \Gamma)$. Note that as  $|x|\to \infty$, the change of $\arg(z-i)$ on the curve $\Gamma([-x, x])$ approaches $2\pi$. Applying 
Corollary~\ref{necessary-mean-planar} to the points $i$,  $\Gamma(-x)$ and $\Gamma(x)$, we obtain a contradiction.
\end{proof}

\section{Local Lipschitz mixers and means on metric circles}\label{sec-mixers-circles} 

Theorem~\ref{intro-quasicircles-mixers} is a special case of the following result. 

\begin{theorem}\label{thm-quasicircles-mixers}
Suppose that $X$ is a metric space homeomorphic to the circle $S^1$. The following properties are related as \eqref{prop-QH}$\implies$\eqref{prop-semilocal}$\implies$\eqref{prop-QS}$\iff$\eqref{prop-local}:  
\begin{enumerate}[(i)]
    \item\label{prop-QH} $X$ has a quasihomogeneous parameterization by $S^1$; 
    \item\label{prop-semilocal} $X$ admits a semilocal Lipschitz mixer;
    \item\label{prop-QS} $X$ has bounded turning; 
    \item\label{prop-local} $X$ admits a local Lipschitz mixer.
\end{enumerate}
\end{theorem}

\begin{remark} The properties listed in Theorem~\ref{thm-quasicircles-mixers} can be quantified as follows: property \eqref{prop-QH} by the function $\eta$ in~\eqref{eq:qh-map}; properties \eqref{prop-semilocal} and ~\eqref{prop-local} by the Lipschitz constant of the mixer and the relative size $r/\diam X$ of its domain (see Definition~\ref{def-local-mixer}); finally, \eqref{prop-QS} can be quantified by the bounded turning constant $C$. The proof of Theorem~\ref{thm-quasicircles-mixers} will show that the forward implications \eqref{prop-QH}$\implies$\eqref{prop-semilocal}$\implies$\eqref{prop-QS}$\implies$\eqref{prop-local} are such that the quantitative data involved in the conclusion depend only on the data in the assumption. However, this is not so for the converse implication \eqref{prop-local}$\implies$\eqref{prop-QS}, as we will see in Example~\ref{no-quant-BT}.
\end{remark}

In the proof of Theorem~\ref{thm-quasicircles-mixers}
we will encounter a slight complication in that the Lipschitz estimate 
\begin{equation}\label{eq-Lip-property}
d(\sigma(a, b, c), \sigma(a', b', c'))\le 
L(d(a, a')+d(b,b')+d(c,c')) 
\end{equation}
for (semi)local mixers does not reduce to the case when only one of three variables changes. 

\begin{proof}[Proof of Theorem~\ref{thm-quasicircles-mixers}] 

\eqref{prop-QH}$\implies$\eqref{prop-semilocal}. By Lemma~\ref{LNR-implies-mixer} the unit circle $S^1\subset \mathbb R^2$ with the metric inherited from $\mathbb R^2$ admits a semilocal Lipschitz mixer. By Proposition~\ref{prop:qhmaps} the same holds for any quasihomogeneous image of $S^1$. 

\eqref{prop-semilocal}$\implies$\eqref{prop-QS}. Suppose that $\sigma\colon \widetilde{\Delta}_R(X^3)\to X$ is an $L$-Lipschitz semilocal mixer.  Given distinct points $a, b\in X$ we consider two cases.  If $d(a, b)\le r$, then the proof of Proposition~\ref{path-conn-BT} yields a curve from $a$ to $b$ of diameter at most $2L\,d(a, b)$. If $d(a, b)>r$, then it suffices to note that any subarc between $a$ and $b$ has diameter at most $\diam X$. Thus, $X$ is of $C$-bounded turning with the constant $C=\max(\diam(X)/r, 2L)$.

\eqref{prop-QS}$\implies$\eqref{prop-local}. Let $r = (\diam X)/(9C)$ where $C$ is the bounded turning constant. Any three points $a, b, c\in X$ with $\diam \{a, b, c\}\le r$ are contained in an arc $\gamma\subset X$ with $\diam \gamma\le C  r \le \frac{1}{9} \diam X$. Define $\sigma(a,b,c)$ to be the median of $a,b,c$ with respect to either orientation of $\gamma$.   

Suppose that $a',b',c'$ is another triple of points with $\diam \{a', b', c'\}\le r$. We may assume that $d(a, a')+d(b,b')+d(c,c')<r$, for otherwise~\eqref{eq-Lip-property} holds with $L = (\diam X)/r$. Since the set $ \{a,b,c,a',b',c'\}$ has diameter at most $3r$, it is contained in an arc $\widetilde{\gamma}\subset X$ with $\diam \widetilde{\gamma} \le \frac{1}{3} \diam X$. Observe that $\widetilde{\gamma}$ is an arc of bounded turning with the same constant $C$. This allows us to apply Proposition~\ref{BT-char} and conclude that $\sigma$ is a Lipschitz map with a constant that depends only on $(\diam X)/r$ and $C$. 

\eqref{prop-local}$\implies$\eqref{prop-QS}. Suppose that $X$ has a local Lipschitz mixer $\sigma\colon \Delta_r(X^3)\to X$ but is not of bounded turning. Then there exist two sequences $\{a_n\}, \{b_n\}$  such that $\diam \Gamma_{n} \ge n\,d(a_n, b_n)$ for any arc $\Gamma_{n}\subset X$ containing both $a_n$ and $b_n$. Since $X$ is bounded, we have $d(a_n, b_n)\to 0$. After passing to a subsequence, $a_n\to p$ and $b_n\to p$ for some $p\in X$. Let $\gamma\subset X$ be an arc which contains $p$ in its interior and satisfies $\diam \gamma\le r$. By Proposition~\ref{path-conn-BT} the arc $\gamma$ has bounded turning, which is a contradiction because $a_n, b_n\in \gamma$ for all sufficiently large $n$. 
\end{proof}

\begin{remark} The mixers constructed in the proofs of \eqref{prop-QH}$\implies$\eqref{prop-semilocal} and \eqref{prop-QS}$\implies$\eqref{prop-local} in Theorem~\ref{thm-quasicircles-mixers} are symmetric. 
\end{remark}

Two examples will demonstrate the failure of converse implications in Theorem~\ref{thm-quasicircles-mixers}. The first of them is based on Rickman's rug~\cite[Example 9.3]{Freeman13}, i.e.,  the product of a fractal with a line. Observe that the product of two spaces $X_1, X_2$ with  Lipschitz mixers $\sigma_1,\sigma_2$ also admits a Lipschitz mixer, namely $\sigma((a_1, a_2), (b_1, b_2))=(\sigma_1(a_1, b_1), \sigma_2(a_2, b_2))$. 

\begin{example}\label{semilocal-no-QH} (\eqref{prop-semilocal}$\notimplies$\eqref{prop-QH}) 
Let $X$ be the real line with the metric defined as $d(x, y)=\sqrt{|x-y|}$ when $|x|, |y|\le 1$ and $d(x, y)=|x-y|$ when either $x, y \ge 1$ or $x,y \le -1$. In the remaining cases the metric is defined by gluing~\cite[\S3.1]{BBI}, i.e., it is the maximal metric compatible with the formulas above. The space $X$ has bounded turning with constant $C=1$. By Corollary~\ref{BT-char-mix} $X$ admits a Lipschitz mixer, therefore so does the product $X\times \mathbb R$.  
 
Let $\Gamma\subset X\times \mathbb R$ be the boundary of the rectangle with vertices $(\pm 3, \pm 2)$. A neighborhood of $\Gamma$ can be retracted onto $\Gamma$ by a Lipschitz map such as 
\[
(x, y) \mapsto \begin{cases}
(x, 2\sign y) & \text{when }|x|\le 1 \\ 
(1, 0) + \frac{2}{\max(|x-1|, |y|)}(x-1, y) & \text{when } x> 1 \\ 
(-1, 0) + \frac{2}{\max(|x+1|, |y|)}(x+1, y) & \text{when } x<-1  \end{cases}
\]
which is illustrated in Figure~\ref{fig-rectangle}. By Lemma~\ref{LNR-implies-mixer} the set $\Gamma$ admits a semilocal Lipschitz mixer. 

Suppose to the contrary that there exists a  quasihomogeneous parameterization $\varphi\colon S^1\to \Gamma$. Conjugating the rotations of $S^1$ by $\varphi$, we obtain a 1-parameter group of bi-Lipschitz maps which acts transitively on $\Gamma$ (see~\cite[Lemma 4.4]{GhamsariHerron} or~\cite{Bishop01}). However, the construction of $\Gamma$ ensures that there is no bi-Lipschitz map $f\colon \Gamma\to \Gamma$ such that $f(3, 0) = (0,2)$. 
\end{example}

\begin{figure}[h]
    \centering
    \includegraphics[width=0.4\textwidth]{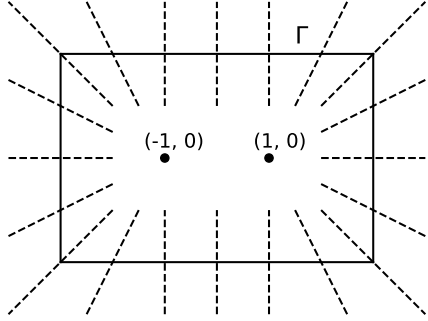}
    \caption{Lipschitz neighborhood retraction in Example~\ref{semilocal-no-QH}}
    \label{fig-rectangle}
\end{figure}

\begin{example}\label{example-vertex-snowflake} (\eqref{prop-QS}$\notimplies$\eqref{prop-semilocal}) Modify the standard construction of the von Koch snowflake~\cite[Example 1.2.10]{BishopPeres} as follows. Let $S_0$ be an equilateral triangle of sidelength $1$. Mark one of its vertices as $p$. To obtain $S_{n+1}$ from $S_n$, choose an arc $\gamma_n\subset S_n$ of length $1$ with an endpoint at $p$ and replace each line segment of that arc with four line segments as in the von Koch construction. The arc $\gamma_n$ consists of $3^n$ line segments of length $3^{-n}$, each of which is replaced by $4$ line segments of length $3^{-n-1}$. The resulting sequence converges to a curve $S$ of bounded turning~\cite{Rohde}. Since each step increases the length by $1/3$, the curve $S$ is unrectifiable. 

For any point $q \in S\setminus \{p\}$ there is $N$ such that $q\notin \gamma_n$  for $n\ge N$. Indeed, otherwise the arc between $p$ and $q$ undergoes the segment replacement infinitely many times, increasing in length indefinitely. This is impossible since the replacement is done only on an arc of length $1$. Thus, $S\setminus \{p\}$ is locally rectifiable. It remains to apply Lemma~\ref{lem-no-semilocal} below.  
\end{example}

\begin{lemma}\label{lem-no-semilocal} Suppose that $\Gamma$ is an unrectifiable metric circle such that $\Gamma\setminus \{p\}$ is locally rectifiable for some point $p\in \Gamma$. Then $\Gamma$ does not have a semilocal Lipschitz mixer.   
\end{lemma}

\begin{proof} Suppose that such a mixer $\sigma \colon \widetilde{\Delta}_r(\Gamma^3)\to \Gamma$ exists. Choose a subarc $\gamma$ containing $p$ in its interior so that $\diam \gamma\le r$ and $\Lip(\sigma) \diam \gamma \le \frac13 \diam (\Gamma \setminus \gamma)$. Let $a, b$ be the endpoints of $\gamma$. Since $\Gamma\setminus \gamma$ is rectifiable, there exists a Lipschitz surjection $f\colon [0, 1]\to \Gamma\setminus \gamma$ such that $f(0)=a$ and $f(1)=b$. Let $g(t)=\sigma(a,b,f(t))$, which is defined because $d(a, b)\le r$. Note that $g$ is a Lipschitz map of $[0, 1]$ into $\Gamma$ with $g(0)=a$ and $g(1)=b$. By~\eqref{eq-mixer-displacement} we have $d(g(t), a) \le \Lip(\sigma) \diam \gamma \le \frac13 \diam (\Gamma \setminus \gamma) $ for all $t\in [0, 1]$. It follows that the range of $g$ does not cover $\Gamma\setminus \gamma$, and therefore must cover $\gamma$. This is a contradiction since $\gamma$ is unrectifiable.
\end{proof}

\begin{remark} Example~\ref{example-vertex-snowflake} highlights the difference between local and semilocal mixers. The metric circle constructed in Example~\ref{example-vertex-snowflake}  admits a local Lipschitz mixer but not a semilocal one. 
\end{remark}

\begin{example}\label{ex-tv-curve} Tukia and V\"ais\"al\"a~\cite[Example 4.12]{TukiaVaisala} gave the following example of a metric arc of bounded turning which is not a quasisymmetric image of $[0, 1]$. Consider $\ell^\infty$, the space of all bounded sequences, treated as functions $x\colon \mathbb N\to \mathbb R$, with the norm $\|x\| = \sup_{k} |x(k)|$.  Define $v_n\in \ell^\infty$ by $v_n(k)=1/n$ when $k=n$ and  $v_n(k)=0$ otherwise. Since $\|v_n\|=1/n$, the piecewise linear curve obtained by connecting each $v_n$ to $v_{n+1}$ tends to $0$. Adding the endpoint $0$ to it, one obtains a metric arc which has bounded turning and therefore admits a Lipschitz mixer. 

By also adding a line segment from $0$ to $v_1$, we get a metric circle $\Gamma$ of bounded turning. Since $\Gamma$ is unrectifiable but $\Gamma\setminus \{0\}$ is locally rectifiable, there is no semilocal Lipschitz mixer on $\Gamma$. 
\end{example}

Example~\ref{no-quant-BT} below will demonstrate that, in contrast to semilocal mixers, a local Lipschitz mixer does not provide quantitative control over the bounded turning property of a closed curve. This is why the proof of \eqref{prop-local}$\implies$\eqref{prop-QS} in Theorem~\ref{thm-quasicircles-mixers} was a nonquantitative compactness
argument. The construction of Example~\ref{no-quant-BT} will use the following elementary lemma.

\begin{lemma}\label{simple-Lip-retract}
Let $I\subset \mathbb R$ be a closed interval, possibly unbounded. Suppose that two functions $\phi\colon I\to \mathbb R$ and $\psi\colon I\to \mathbb R$ are $L$-Lipschitz and satisfy $\phi(x)\le \psi(x)$ for all $x\in I$. Then there is a Lipschitz retraction $\Phi$ of $\mathbb R^2$ onto the set $E = \{(x, y)\colon x\in I, \phi(x)\le y\le \psi(x)\}$. Moreover, $\Lip \Phi\le \sqrt{L^2+1}$. 
\end{lemma}

\begin{proof}  Extend the identity map on $I$ to a $1$-Lipschitz function $f\colon \mathbb R\to I$ be the function that is locally constant on $\mathbb R\setminus I$. Then define 
\[
g(x,y) = (x, \med(y, \phi(x), \psi(x))),\quad (x, y)\in I\times \mathbb R.
\]
Since the second component of $g$ is $L$-Lipschitz, we have $\Lip g\le \sqrt{L^2+1}$. By construction $g$ is the identity map on $E$. Finally, the desired retraction $\Phi\colon \mathbb R^2\to E$ is obtained as $\Phi(x,y)=g(f(x),y)$. 
\end{proof}

\begin{example}\label{no-quant-BT} 
Define $g\colon \mathbb R^2\to \mathbb R$ by $g(x, y)=\big||x|-1\big|+|y|$. For $t>0$ let $G_t$ be the open set $\{(x, y)\colon g(x, y)<t\}$. 
When $1<t<2$, the Jordan curve $\Gamma_t = \partial G_t$ admits a $50$-Lipschitz local mixer with the domain $\Delta_{1/6}(\Gamma_t^3)$. However, its constant of bounded turning  is unbounded as $t\to 1^+$. 
\end{example}

\begin{proof} The second claim follows from the fact that the points $(0, \pm (t-1))$ divide $\Gamma_t$ into two symmetric parts of diameter $2t$.  

Let $\sigma_m$ be the median mixer~\eqref{eq-coordinate-median} on $\mathbb R^2$. 
If a triple of points $(a,b, c)$ belongs to $\Delta_{1/6}(\Gamma_t^3)$, then it is contained either in the set $U=\{(x,y)\colon |x|\ge 1\}$ or in the set $V=\{(x,y)\colon |y|\ge t-1\}$. 
In either case, the point $w:=\sigma_m(a,b,c)$ belongs to $U\cup V$. Furthermore, we have $w\notin G_{1/2}$ because
$\dist(w, \Gamma_t)\le 1/3$ by virtue of~\eqref{eq-mixer-displacement}. In conclusion, 
\begin{equation}\label{eq-range-sigma-m}
\sigma_m(\Delta_{1/6}(\Gamma_t^3))\subset (E\cup F)\setminus G_{1/2}.
\end{equation}
The set $E\cup F$ is the complement of  the rectangle $Q$ with vertices $(\pm 1, \pm (t-1))$. See Figure~\ref{fig-projection}. To complete the proof, we need a Lipschitz retraction $F\colon (E\cup F)\setminus G_{1/2}\to \Gamma_t$. The required local mixer will be $F\circ \sigma_m$ restricted to $\Delta_{1/6}(\Gamma_t^3)$.  

\begin{figure}[h]
    \centering
    \includegraphics[width=0.5\textwidth]{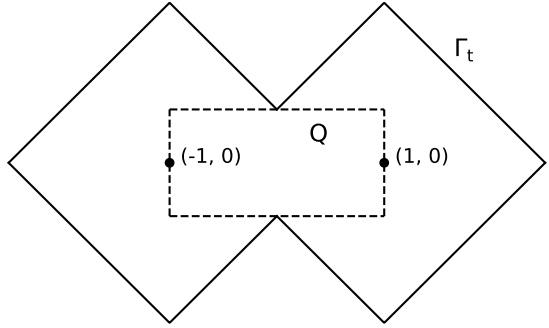}
    \caption{The curve $\Gamma_t$ of Example~\ref{no-quant-BT}}
    \label{fig-projection}
\end{figure}

Since $\partial G_t$ consists of line segments with slope $\pm 1$, by Lemma~\ref{simple-Lip-retract} there exists a retraction $\Phi\colon \mathbb R^2\to \overline{G_t}$ with $\Lip \Phi \le \sqrt{2}$. Let 
\begin{equation}\label{eq-def-F}
F(x, y) = \begin{cases}
(1,0) + \frac{t}{g(x,y)}(x-1,y), \quad & 
(x, y)\in G_t \text{ and }  x>0; \\ 
(-1,0) + \frac{t}{g(x,y)}(x+1,y),\quad & 
(x, y)\in G_t \text{ and }  x<0; \\ 
\Phi(x,y),\quad & (x,y)\notin G_t.    
\end{cases}
\end{equation}

By construction, $F$ is the identity map on $\Gamma_t$. That its range is precisely $\Gamma_t$ is clear from Figure~\ref{fig-projection}: the radial projection with respect to a point $(\pm 1, 0)$ maps each half of the set $G_t\setminus Q$ to a half of the curve $\Gamma_t$.

It remains to give a uniform bound on $\Lip F$. Since $\Lip \Phi\le \sqrt{2}$, it suffices to consider $F$ on $G_t\setminus G_{1/2}$, and more specifically on the right half of this set, where $x\ge 0$.  
Since $\Lip g = \sqrt{2}$ and $g\ge 1/2$ on the complement of $G_{1/2}$, the function 
$(x,y)\mapsto t/g(x,y)$ has the Lipschitz constant at most $t\sqrt{2}/(1/2)^2 \le 8\sqrt{2}$. It is also bounded by $2t\le 4$. The vector function $(x,y)\mapsto (x-1,y)$ is $1$-Lipschitz and is bounded by $t\le 2$ on $G_t$. This yields
$\Lip F\le (8\sqrt{2})\cdot 2 + 1\cdot 4 < 50$. 
\end{proof}

The section concludes with a brief discussion of local Lipschitz means on metric circles. 

\begin{proposition}\label{prop-quasicircles-means}
If $X$ is a metric circle of $C$-bounded turning, then there exists an $L$-Lipschitz mean $\mu\colon \Delta_r(X^2)\to X$ where $L$ and $r/\diam X$ depend only on $C$.
\end{proposition}

\begin{proof} Choose an orientation on $X$. Let $r = (\diam X)/(9C)$. Any two points $a, b\in X$ with $d(a,b) \le r$ are contained in an arc $\gamma\subset X$ with $\diam \gamma\le C  r \le \frac{1}{9} \diam X$. Let $\mu(a,b) = \min(a, b)$ with respect to the order on $\gamma$ induced by the orientation of $X$. The Lipschitz continuity of $\mu$ follows in the same way as in the proof of \eqref{prop-QS}$\implies$\eqref{prop-local} in  Theorem~\ref{thm-quasicircles-mixers}.
\end{proof}

In contrast to  Theorem~\ref{thm-quasicircles-mixers}, the existence of a local Lipschitz mean on a metric circle does not necessarily make it a curve of bounded turning.  

\begin{example}\label{mean-no-qs} 
Let $\Gamma=\gamma_1\cup \gamma_2$ be the Jordan curve formed by the graph $y=\sqrt{|x|}$, $-1\le x\le 1$, denoted $\gamma_1$, and the line segment from $(1, 1)$ to $(-1, 1)$, denoted $\gamma_2$.  
Due to the cusp at $(0, 0)$, this curve does not have bounded turning. However, it has a local Lipschitz mean with the domain $\Delta_{1}(\Gamma^2)$. 
\end{example}

\begin{proof} Let $C =\gamma_1\cap \gamma_2= \{(-1, 1), (1, 1)\}$. 
Let $\mu$ be the Lipschitz mean on $\gamma_1$ from the proof of Corollary~\ref{cor-two-quasiarcs}, that is, the map whose action on the $x$-coordinates is \[
(x_1, x_2) \mapsto \left(\min(|x_1|, |x_2|)\right) \sign \max (x_1, x_2). 
\]
Note that $\mu(a, b)=a$ when $a\in \gamma_1$, $b\in C$, and $|a-b|\le 1$. Extend $\mu$ by letting $\mu(a,b)=\mu(b,a)=a$ when $a\in \gamma_1$ and $b\in \gamma_2$ with $|a-b|\le 1$. This is clearly a Lipschitz function. As a special case, we have $\mu(a, b) = b$ when $a\in \gamma_2$, $b\in C$, and $|a-b|\le 1$. 

In the remaining case of $a,b\in \gamma_2$ with $|a-b|\le 1$, define 
\begin{equation}\label{eq-segment-mean}
\mu(a, b) = \frac{\dist(b, C) a + \dist(a, C) b }{\dist(a, C) + \dist(b, C)}
\end{equation}
when $a\ne b$, and $\mu(a, a)=a$. 
The denominator of~\eqref{eq-segment-mean} is at least $1$, which implies that the quotient is a Lipschitz function. The formula~\eqref{eq-segment-mean} is symmetric and satisfies $\mu(a, b)\to a$ as $b\to a$, which is consistent with the property $\mu(a,a)=a$. Finally,~\eqref{eq-segment-mean} is consistent with the property $\mu(a, b) = b$ when $a\in \gamma_2$, $b\in C$, and $|a-b|\le 1$. 
\end{proof}

\section{Questions and remarks}\label{sec-remarks}

\begin{question}\label{q-semilocal-param} If $\Gamma$ is a metric circle with a semilocal Lipschitz mixer, does it follow that $\Gamma$ is a quasisymmetric image of $S^1$, i.e., a \emph{quasicircle}?  
\end{question}

Since we know that $\Gamma$ has bounded turning, by the Tukia-Vaisala theorem~\cite[\S4]{TukiaVaisala} the question reduces to whether $\Gamma$ is doubling (Definition~\ref{def-metric-doubling}). The non-doubling curve of bounded turning in Example~\ref{ex-tv-curve} is not a counterexample since it does not admit a semilocal Lipschitz mixer.  

A theorem of Meyer~\cite{Meyer} implies that the existence of a \emph{local} Lipschitz mixer on $\Gamma$ is equivalent to having a \emph{weakly quasisymmetric} parameterization by $S^1$, see~\cite{Meyer} for the definition.

\begin{question}\label{q-mixer-mean} Does every metric space with a Lipschitz mixer also admit a Lipschitz mean? 
\end{question}

The straightforward attempt to construct a mean $\mu$ from a mixer $\sigma$, by letting $\mu(a, b)=\sigma(a, b, c)$ for some fixed $c$, succeeds only if $\sigma$ is a symmetric mixer (or at least is symmetric with respect to the variables $a, b$).

\begin{question}\label{q-lcp-mixer} Does every metric space with a Lipschitz clustering property (Definition~\ref{def-lcp}) admit a Lipschitz mixer?  
\end{question}

We know from Lemma~\ref{set-symm-means} that a Lipschitz retraction $X(2)\to X$ provides a Lipschitz mean. However, a Lipschitz retraction $X(3)\to X$ does not guarantee the existence of a Lipschitz mixer. Indeed, the parabola  $y=x^2$ admits Lipschitz retractions $X(n)\to X$ for all $n$ (Remark~\ref{rem-not-BT}) but does not have a Lipschitz mixer (Corollary~\ref{BT-char-mix}). This example falls short of answering Question~\ref{q-lcp-mixer} since a parabola does not have the Lipschitz clustering property~\cite[Example 3]{Kovalev2022}.

\bibliography{references.bib} 
\bibliographystyle{plain} 

\end{document}